\documentclass{amsart}
\usepackage{amsmath,amssymb,bm}
\usepackage{amsmath,amsthm,amssymb,cases}
\usepackage[all]{xy}
\usepackage{amsmath}
\usepackage{amssymb}
\usepackage{amscd}
\usepackage{amsthm}
\usepackage[dvips]{graphicx}
\usepackage{color}

\theoremstyle{definition}
\newtheorem{defn}{Definition}[section]

\newtheorem{rem}[defn]{Remark}
\theoremstyle{plain}
\newtheorem{thm}[defn]{Theorem}

\newtheorem{lem}[defn]{Lemma}

\numberwithin{equation}{section}
\allowdisplaybreaks
\title[Presentation for mapping class group]{An infinite presentation for the mapping class group of a non-orientable surface}

\author[G.~Omori]{Genki Omori}
\address{
(Genki Omori)
Department of Mathematics,
Tokyo Institute of Technology,
Oh-okayama, Meguro, Tokyo 152-8551, Japan
}

\email{omori.g.aa@m.titech.ac.jp}
\date{\today}
\begin{document}
\maketitle
\begin{abstract}
We give an infinite presentation for the mapping class group of a non-orientable surface. The generating set consists of all Dehn twists and all crosscap pushing maps along simple loops. 
\end{abstract}

\section{Introduction}

Let $\Sigma _{g,n}$ be a compact connected orientable surface of genus $g\geq 0$ with $n\geq 0$ boundary components. The {\it mapping class group} $\mathcal{M}(\Sigma _{g,n})$ of $\Sigma _{g,n}$ is the group of isotopy classes of orientation preserving self-diffeomorphisms on $\Sigma _{g,n}$ fixing the boundary pointwise. A finite presentation for $\mathcal{M}(\Sigma _{g,n})$ was given by Hatcher-Thurston~\cite{Hatcher-Thurston}, Wajnryb~\cite{Wajnryb}, Harer~\cite{Harer}, Gervais~\cite{Gervais2} and Labru\`ere-Paris~\cite{Labruere-Paris}. Gervais~\cite{Gervais} obtained an infinite presentation for $\mathcal{M}(\Sigma _{g,n})$ by using Wajnryb's finite presentation for $\mathcal{M}(\Sigma _{g,n})$, and Luo~\cite{Luo} rewrote Gervais' presentation into a simpler infinite presentation (see Theorem~\ref{pres_Gervais}).

Let $N_{g,n}$ be a compact connected non-orientable surface of genus $g\geq 1$ with $n\geq 0$ boundary components. The surface $N_g=N_{g,0}$ is a connected sum of $g$ real projective planes. The mapping class group $\mathcal{M}(N_{g,n})$ of $N_{g,n}$ is the group of isotopy classes of self-diffeomorphisms on $N_{g,n}$ fixing the boundary pointwise. For $g\geq 2$ and $n\in \{ 0,1\}$, a finite presentation for $\mathcal{M}(N_{g,n})$ was given by Lickorish~\cite{Lickorish1}, Birman-Chillingworth~\cite{Birman-Chillingworth}, Stukow~\cite{Stukow1} and Paris-Szepietowski~\cite{Paris-Szepietowski}. Note that $\mathcal{M}(N_1)$ and $\mathcal{M}(N_{1,1})$ are trivial (see \cite[Theorem~3.4]{Epstein}) and $\mathcal{M}(N_2)$ is finite (see \cite[Lemma~5]{Lickorish1}). Stukow~\cite{Stukow2} rewrote Paris-Szepietowski's presentation into a finite presentation with Dehn twists and a ``Y-homeomorphism'' as generators (see Theorem~\ref{thm_Stukow}).

In this paper, we give a simple infinite presentation for $\mathcal{M}(N_{g,n})$ (Theorem~\ref{main-thm}) when $g\geq 1$ and $n\in \{ 0,1\}$. The generating set consists of all Dehn twits and all ``crosscap pushing maps'' along simple loops. We review the crosscap pushing map in Section~\ref{Preliminaries}. We prove Theorem~\ref{main-thm} by applying Gervais' argument to Stukow's finite presentation.

\section{Preliminaries}\label{Preliminaries}

\subsection{Relations among Dehn twists and Gervais' presentation}
Let $S$ be either $N_{g,n}$ or $\Sigma _{g,n}$. We denote by $\mathcal{N}_{S}(A)$ a regular neighborhood of a subset $A$ in $S$ . 
For every simple closed curve $c$ on $S$, we choose an orientation of $c$ and fix it throughout this paper. However, for simple closed curves $c_1$, $c_2$ on $S$ and $f\in \mathcal{M}(S)$, $f(c_1)=c_2$ means $f(c_1)$ is isotopic to $c_2$ or the inverse curve of $c_2$. If $S$ is a non-orientable surface, we also fix an orientation of $\mathcal{N}_{S}(c)$ for each two-sided simple closed curve $c$. For a two-sided simple closed curve $c$ on $S$, denote by $t_c$ the right-handed Dehn twist along $c$ on $S$. In particular, for a given explicit two-sided simple closed curve, an arrow on a side of the simple closed curve indicates the direction of the Dehn twist (see Figure~\ref{dehntwist}).

\begin{figure}[h]
\includegraphics[scale=0.65]{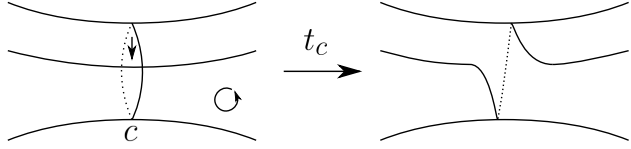}
\caption{The right-handed Dehn twist $t_c$ along a two-sided simple closed curve $c$ on $S$.}\label{dehntwist}
\end{figure}

Recall the following relations on $\mathcal{M}(S)$ among Dehn twists along two-sided simple closed curves on $S$.

\begin{lem}\label{trivial}
For a two-sided simple closed curve $c$ on $S$ which bounds a disk or a M\"{o}bius band in $S$, we have $t_c=1$ on $\mathcal{M}(S)$.
\end{lem} 

\begin{lem}[The braid relation (i)]\label{braid1}
For  a two-sided simple closed curve $c$ on $S$ and $f\in \mathcal{M}(S)$, we have 
\[
t_{f(c)}^{\varepsilon _{f(c)}}=ft_cf^{-1},
\]
where $\varepsilon _{f(c)}=1$ if the restriction $f|_{\mathcal{N}_{S}(c)}:\mathcal{N}_{S}(c)\rightarrow \mathcal{N}_{S}(f(c))$ is orientation preserving and $\varepsilon _{f(c)}=-1$ if the restriction $f|_{\mathcal{N}_{S}(c)}:\mathcal{N}_{S}(c)\rightarrow \mathcal{N}_{S}(f(c))$ is orientation reversing.
\end{lem}

When $f$ in Lemma~\ref{braid1} is a Dehn twist $t_d$ along a two-sided simple closed curve $d$ and the geometric intersection number $|c\cap d|$ of $c$ and $d$ is $m$, we denote by $T_m$ the braid relation.

Let $c_1$, $c_2$, $\dots $, $c_k$ be two-sided simple closed curves on $S$. The sequence $c_1$, $c_2$, $\dots $, $c_k$ of simple closed curves on $S$ is a {\it k-chain on $S$} if $c_1$, $c_2$, $\dots $, $c_k$ satisfy $|c_i\cap c_{i+1}|=1$ for each $i=1$, $2$, $\dots $, $k-1$ and $|c_i\cap c_j|=0$ for $|j-i|>1$.

\begin{lem}[The $k$-chain relation]\label{chain}
Let $c_1$, $c_2$, $\dots $, $c_k$ be a $k$-chain on $S$ and let $\delta _1$, $\delta _2$ (resp. $\delta $) be distinct boundary components (resp. the boundary component) of $\mathcal{N}_S(c_1\cup c_2\cup \cdots \cup c_k)$ when $k$ is odd (resp. even). Then we have
\begin{eqnarray*}
(t_{c_1}^{\varepsilon _{c_1}}t_{c_2}^{\varepsilon _{c_2}}\cdots t_{c_k}^{\varepsilon _{c_k}})^{k+1} &=& t_{\delta _1}^{\varepsilon _{\delta _1}}t_{\delta _2}^{\varepsilon _{\delta _2}} \hspace{0.5cm}\text{when}\ k\text{ is odd},\\
(t_{c_1}^{\varepsilon _{c_1}}t_{c_2}^{\varepsilon _{c_2}}\cdots t_{c_k}^{\varepsilon _{c_k}})^{2k+2} &=& t_\delta ^{\varepsilon _\delta } \hspace{0.5cm}\text{when}\ k\text{ is even},
\end{eqnarray*}
where $\varepsilon _{c_1}$, $\varepsilon _{c_2}$, $\dots $, $\varepsilon _{c_k}$, $\varepsilon _{\delta _1}$, $\varepsilon _{\delta _2}$ and $\varepsilon $ are $1$ or $-1$, and $t_{c_1}^{\varepsilon _{c_1}}$, $t_{c_2}^{\varepsilon _{c_2}}$, $\dots $, $t_{c_k}^{\varepsilon _{c_k}}$, $t_{\delta _1}^{\varepsilon _{\delta _1}}$ and $t_{\delta _2}^{\varepsilon _{\delta _2}}$ (resp. $t_\delta ^{\varepsilon _\delta }$) are right-handed Dehn twists for some orientation of $\mathcal{N}_S(c_1\cup c_2\cup \cdots \cup c_k)$.
\end{lem}

\begin{lem}[The lantern relation]\label{lantern}
Let $\Sigma $ be a subsurface of $S$ which is diffeomorphic to $\Sigma _{0,4}$ and let $\delta _{12}$, $\delta _{23}$, $\delta _{13}$, $\delta _1$, $\delta _2$, $\delta _3$ and $\delta _4$ be simple closed curves on $\Sigma $ as in Figure~\ref{lantern1}. Then we have
\[
t_{\delta _{12}}^{\varepsilon_{\delta _{12}}}t_{\delta _{23}}^{\varepsilon_{\delta _{23}}}t_{\delta _{13}}^{\varepsilon_{\delta _{13}}}=t_{\delta _1}^{\varepsilon_{\delta _1}}t_{\delta _2}^{\varepsilon_{\delta _2}}t_{\delta _3}^{\varepsilon_{\delta _3}}t_{\delta _4}^{\varepsilon_{\delta _4}},
\]
where $\varepsilon _{\delta _{12}}$, $\varepsilon _{\delta _{23}}$, $\varepsilon _{\delta _{13}}$, $\varepsilon_{\delta _1}$, $\varepsilon_{\delta _2}$, $\varepsilon_{\delta _3}$ and $\varepsilon_{\delta _4}$ are $1$ or $-1$, and $t_{\delta _{12}}^{\varepsilon_{\delta _{12}}}$, $t_{\delta _{23}}^{\varepsilon_{\delta _{23}}}$, $t_{\delta _{13}}^{\varepsilon_{\delta _{13}}}$, $t_{\delta _1}^{\varepsilon_{\delta _1}}$, $t_{\delta _2}^{\varepsilon _{\delta _2}}$, $t_{\delta _3}^{\varepsilon _{\delta _3}}$ and $t_{\delta _4}^{\varepsilon _{\delta _4}}$ are right-handed Dehn twists for some orientation of $\Sigma $.
\end{lem}

\begin{figure}[h]
\includegraphics[scale=0.75]{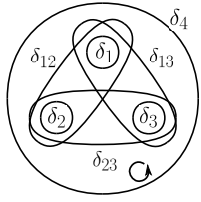}
\caption{Simple closed curves $\delta _{12}$, $\delta _{23}$, $\delta _{13}$, $\delta _1$, $\delta _2$, $\delta _3$ and $\delta _4$ on $\Sigma $.}\label{lantern1}
\end{figure}

Luo's presentation for $\mathcal{M}(\Sigma _{g,n})$, which is an improvement of Gervais' one, is as follows.

\begin{thm}[\cite{Gervais}, \cite{Luo}]\label{pres_Gervais}
For $g\geq 0$ and $n\geq 0$, $\mathcal{M}(\Sigma _{g,n})$ has the following presentation:

generators: $\{ t_c \mid c:\text{ s.c.c. on }\Sigma _{g,n} \}$.

relations:
\begin{enumerate}
 \item[(0$^\prime $)] $t_c=1$ when $c$ bounds a disk in $\Sigma _{g,n}$,
 \item[(I$^\prime $)] All the braid relations $T_0$ and $T_1$,
 \item[(I\hspace{-0.06cm}I)] All the 2-chain relations,
 \item[(I\hspace{-0.06cm}I\hspace{-0.06cm}I)] All the lantern relations.
\end{enumerate} 
\end{thm}

\subsection{Relations among the crosscap pushing maps and Dehn twists}
Let $\mu $ be a one-sided simple closed curve on $N_{g,n}$ and let $\alpha $ be a simple closed curve on $N_{g,n}$ such that $\mu $ and $\alpha $ intersect transversely at one point. Recall that $\alpha $ is oriented. For these simple closed curves $\mu$ and $\alpha $, we denote by $Y_{\mu , \alpha }$ a self-diffeomorphism on $N_{g,n}$ which is described as the result of pushing the M\"{o}bius band $\mathcal{N}_{N_{g,n}}(\mu )$ once along $\alpha $.  We call $Y_{\mu , \alpha }$ a {\it crosscap pushing map}. In particular, if $\alpha $ is two-sided, we call $Y_{\mu , \alpha }$ a {\it Y-homeomorphism} (or {\it crosscap slide}), where a {\it crosscap} means a M\"{o}bius band in the interior of a surface. The Y-homeomorphism was originally defined by Lickorish~\cite{Lickorish1}. We have the following fundamental relation on $\mathcal{M}(N_{g,n})$ and we also call the relation the {\it braid relation}.

\begin{lem}[The braid relation (ii)]\label{braid2}
Let $\mu $ be a one-sided simple closed curve on $N_{g,n}$ and let $\alpha $ be a simple closed curve on $N_{g,n}$ such that $\mu $ and $\alpha $ intersect transversely at one point. For $f\in \mathcal{M}(N_{g,n})$, we have
\[
Y_{f(\mu ),f(\alpha )}^{\varepsilon _{f(\alpha )}}=fY_{\mu ,\alpha }f^{-1},
\]
where $\varepsilon _{f(\alpha )}=1$ if the fixed orientation of $f(\alpha )$ coincides with that induced by the orientation of $\alpha $, and $\varepsilon _{f(\alpha )}=-1$ otherwise.
\end{lem}

We describe crosscap pushing maps from a different point of view. Let $e:D^\prime \hookrightarrow {\rm int}S$ be a smooth embedding of the unit disk $D^\prime \subset \mathbb C$. Put $D:=e(D^\prime )$. Let $S^\prime $ be the surface obtained from $S-{\rm int}D$ by the identification of antipodal points of $\partial D$. We call the manipulation that gives $S^\prime $ from $S$ the {\it blowup of} $S$ {\it on} $D$. Note that the image $M\subset S^\prime$ of $\mathcal{N}_{S-{\rm int}D}(\partial D)\subset S-{\rm int}D$ with respect to the blowup of $S$ on $D$ is a crosscap. Conversely, the {\it blowdown of} $S^\prime$ {\it on }$M$ is the following manipulation that gives $S$ from $S^\prime $. We paste a disk on the boundary obtained by cutting $S$ along the center line $\mu $ of $M$. The blowdown of $S^\prime $ on $M$ is the inverse manipulation of the blowup of $S$ on $D$.

Let $\mu $ be a one-sided simple closed curve on $N_{g,n}$. Note that we obtain $N_{g-1,n}$ from $N_{g,n}$ by the blowdown of $N_{g,n}$ on $\mathcal{N}_{N_{g,n}}(\mu )$. Denote by $x_\mu $ the center point of a disk $D_\mu $ that is pasted on the boundary obtained by cutting $S$ along $\mu $. Let $e:D^\prime \hookrightarrow D_\mu \subset N_{g-1,n}$ be a smooth embedding of the unit disk $D^\prime \subset \mathbb C$ to $N_{g-1,n}$ such that $D_\mu =e(D^\prime )$ and $e(0)=x_\mu $. Let $\mathcal{M}(N_{g-1,n},x_\mu )$ be the group of isotopy classes of self-diffeomorphisms on $N_{g-1,n}$ fixing the boundary $\partial N_{g-1,n}$ and the point $x_\mu $, where isotopies also fix the boundary $\partial N_{g-1,n}$ and $x_\mu $. Then we have the {\it blowup homomorphism} 
\[
\varphi _\mu :\mathcal{M}(N_{g-1,n},x_\mu )\rightarrow \mathcal{M}(N_{g,n})
\]
that is defined as follows. For $h \in \mathcal{M}(N_{g-1,n},x_\mu )$, we take a representative $h^\prime $ of $h$ which satisfies either of the following conditions: (a) $h^\prime |_{D_\mu }$ is the identity map on $D_\mu $, (b) $h^\prime (x)=e(\overline{e^{-1}(x)})$ for $x\in D_\mu $, where $\overline{e^{-1}(x)}$ is the complex conjugation of $e^{-1}(x)\in \mathbb C$. Such $h^\prime $ is compatible with the blowup of $N_{g-1,n}$ on $D_\mu $, thus $\varphi _\mu (h)\in \mathcal{M}(N_{g,n})$ is induced and well defined (c.f. \cite[Subsection~2.3]{Szepietowski1}). 

The {\it point pushing map} 
\[
j_{x_\mu }:\pi _1(N_{g-1,n},x_\mu )\rightarrow \mathcal{M}(N_{g-1,n},x_\mu )
\]
is a homomorphism that is defined as follows. For $\gamma \in \pi _1(N_{g-1,n},x_\mu )$, $j_{x_\mu }(\gamma )\in \mathcal{M}(N_{g-1,n},x_\mu )$ is described as the result of pushing the point $x_\mu $ once along $\gamma $. The point pushing map comes from the Birman exact sequence. Note that for $\gamma _1$, $\gamma _2\in \pi _1(N_{g-1,n})$, $\gamma _1\gamma _2$ means $\gamma _1\gamma _2(t)=\gamma _2(2t)$ for $0\leq t\leq \frac{1}{2}$ and $\gamma _1\gamma _2(t)=\gamma _1(2t-1)$ for $\frac{1}{2}\leq t\leq 1$.

Following Szepietowski~\cite{Szepietowski1} we define the composition of the homomorphisms:
\[
\psi _{x_\mu }:=\varphi _\mu \circ j_{x_\mu }:\pi _1(N_{g-1,n},x_\mu )\rightarrow \mathcal{M}(N_{g,n}).
\]
For each closed curve $\alpha $ on $N_{g,n}$ which transversely intersects with $\mu $ at one point, we take a loop $\overline{\alpha }$ on $N_{g-1,n}$ based at $x_\mu $ such that $\overline{\alpha }$ has no self-intersection points on $D_\mu $ and $\alpha $ is the image of $\overline{\alpha }$ with respect to the blowup of $N_{g-1,n}$ on $D_\mu $. If $\alpha $ is simple, we take $\overline{\alpha }$ as a simple loop. The next two lemmas follow from the description of the point pushing map (see \cite[Lemma~2.2, Lemma~2.3]{Korkmaz2}). 

\begin{lem}\label{pushing1}
For a simple closed curve $\alpha $ on $N_{g,n}$ which transversely intersects with a one-sided simple closed curve $\mu $ on $N_{g,n}$ at one point, we have
\[
\psi _{x_\mu }(\overline{\alpha })=Y_{\mu ,\alpha }.
\]
\end{lem}

\begin{lem}\label{pushing2}
For a one-sided simple closed curve $\alpha $ on $N_{g,n}$ which transversely intersects with a one-sided simple closed curve $\mu $ on $N_{g,n}$ at one point, we take $\mathcal{N}_{N_{g-1,n}}(\overline{\alpha })$ such that the interior of $\mathcal{N}_{N_{g-1,n}}(\overline{\alpha })$ contains $D_\mu $. Suppose that $\overline{\delta _1}$ and $\overline{\delta _2}$ are distinct boundary components of $\mathcal{N}_{N_{g-1,n}}(\overline{\alpha })$, and $\delta _1$ and $\delta _2$ are two-sided simple closed curves on $N_{g,n}$ which are image of $\overline{\delta _1}$, $\overline{\delta _2}$ with respect to the blowup of $N_{g-1,n}$ on $D_\mu $, respectively. Then we have 
\[
Y_{\mu ,\alpha }=
t_{\delta _1}^{\varepsilon _{\delta _1}}t_{\delta _2}^{\varepsilon _{\delta _2}},
\] 
where $\varepsilon _{\delta _1}$ and $\varepsilon _{\delta _2}$ are $1$ or $-1$, and $\varepsilon _{\delta _1}$ and $\varepsilon _{\delta _2}$ depend on the orientations of $\alpha$, $\mathcal{N}_{N_{g,n}}(\delta _1)$ and $\mathcal{N}_{N_{g,n}}(\delta _2)$ (see Figure~\ref{crosscap_def_twist}). 
\end{lem}

\begin{figure}[h]
\includegraphics[scale=0.6]{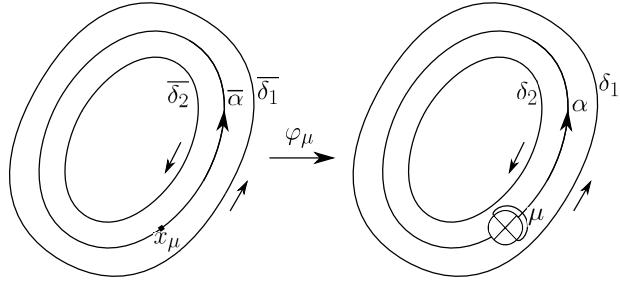}
\caption{If the orientations of $\alpha$, $\mathcal{N}_{N_{g,n}}(\delta _1)$ and $\mathcal{N}_{N_{g,n}}(\delta _2)$ are as above, then we have $Y_{\mu ,\alpha }=t_{\delta _1}t_{\delta _2}^{-1}$. The x-mark means that antipodal points of $\partial D_\mu $ are identified.}\label{crosscap_def_twist}
\end{figure}

By the definition of the homomorphism $\psi _{x_\mu }$ and Lemma~\ref{pushing1}, we have the following lemma.

\begin{lem}\label{pushing3}
Let $\alpha $ and $\beta $ be simple closed curves on $N_{g,n}$ which transversely intersect with a one-sided simple closed curve $\mu $ on $N_{g,n}$ at one point each. Suppose the product $\overline{\alpha }\overline{\beta }$ of $\overline{\alpha }$ and $\overline{\beta }$ in $\pi _1(N_{g-1,n},x_\mu )$ is represented by a simple loop on $N_{g-1,n}$, and $\alpha \beta $ is a simple closed curve on $N_{g,n}$ which is the image of the representative of $\overline{\alpha }\overline{\beta }$ with respect to the blowup of $N_{g-1,n}$ on $D_\mu $. Then we have
\[
Y_{\mu ,\alpha \beta }=Y_{\mu ,\alpha }Y_{\mu ,\beta }.
\]
\end{lem}

Finally, we recall the following relation between a Dehn twist and a Y-homeomorphism.

\begin{lem}\label{pushing4}
Let $\alpha $ be a two-sided simple closed curve on $N_{g,n}$ which transversely intersect with a one-sided simple closed curve $\mu $ on $N_{g,n}$ at one point and let $\delta $ be the boundary of $\mathcal{N}_{N_{g,n}}(\alpha \cup \mu )$. Then we have
\[
Y_{\mu ,\alpha }^2=t_\delta ^\varepsilon ,
\]
where $\varepsilon $ is $1$ or $-1$, and $\varepsilon $ depends on the orientations of $\alpha$ and $\mathcal{N}_{N_{g,n}}(\delta )$ (see Figure~\ref{yhomeo_square}). 
\end{lem}
Lemma~\ref{pushing4} follows from relations in Lemma~\ref{trivial}, Lemma~\ref{pushing2} and Lemma~\ref{pushing3}.

\begin{figure}[h]
\includegraphics[scale=0.6]{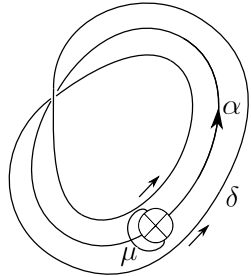}
\caption{If the orientations of $\alpha$ and $\mathcal{N}_{N_{g,n}}(\delta )$ are as above, then we have $Y_{\mu ,\alpha }^2=t_{\delta _1}$.}\label{yhomeo_square}
\end{figure}

\subsection{Stukow's finite presentation for $\mathcal{M}(N_{g,n})$}\label{section_Stukow_pres}

Let $e_i:D^\prime \hookrightarrow \Sigma _0$ for $i=1$, $2, \dots $, $g+1$ be smooth embeddings of the unit disk $D^\prime \subset \mathbb C$ to a 2-sphere $\Sigma _0$ such that $D_i:=e_i(D^\prime )$ and $D_j$ are disjoint for distinct $1\leq i,j\leq g+1$. Then we take a model of $N_g$ (resp. $N_{g,1}$) as the surface obtained from $\Sigma _0$ (resp. $\Sigma _0-{\rm int}D_{g+1}$) by the blowups on $D_1,\dots ,D_g$ and we describe the identification of $\partial D_i$ by the x-mark as in Figures~\ref{scc_gamma} and \ref{sccs}. When $n\in \{ 0,1\}$, for $1\leq i_1<i_2<\dots <i_k \leq g$, let $\gamma _{i_1,i_2,\dots ,i_k}$ be the simple closed curve on $N_{g,n}$ as in Figure~\ref{scc_gamma}. Then we define the simple closed curves $\alpha _i:=\gamma _{i,i+1}$ for $i=1$, $\dots $, $g-1$, $\beta :=\gamma _{1,2,3,4}$ and $\mu _1:=\gamma _1$ (see Figure~\ref{sccs}), and the mapping classes $a_i:=t_{\alpha _i}$ for $i=1$, $\dots $, $g-1$, $b:=t_\beta $ and $y:=Y_{\mu _1,\alpha _1}$. Then the following finite presentation for $\mathcal{M}(N_{g,n})$ is obtained by Lickorish~\cite{Lickorish1} for $(g,n)=(2,0)$, Stukow~\cite{Stukow1} for $(g,n)=(2,1)$, Birman-Chillingworth~\cite{Birman-Chillingworth} for $(g,n)=(3,0)$ and Theorem~3.1 and Proposition~3.3 in \cite{Stukow2} for the other $(g,n)$ such that $g\geq 3$ and $n\in \{ 0,1\}$. 

\begin{figure}[h]
\includegraphics[scale=0.7]{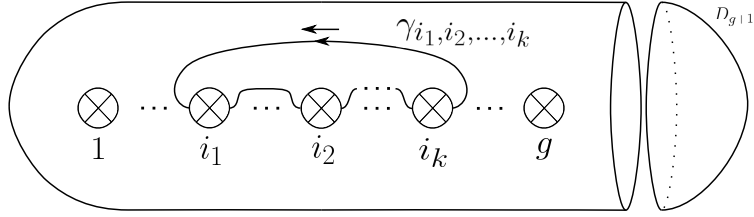}
\caption{Simple closed curve $\gamma _{i_1,i_2,\dots ,i_k}$ on $N_{g,n}$.}\label{scc_gamma}
\end{figure}

\begin{figure}[h]
\includegraphics[scale=0.7]{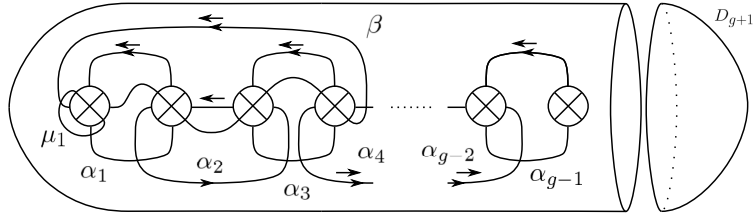}
\caption{Simple closed curves $\alpha _1$, $\dots $, $\alpha _{g-1}$, $\beta $ and $\mu _1$ on $N_{g,n}$.}\label{sccs}
\end{figure}

\begin{thm}[\cite{Lickorish1}, \cite{Birman-Chillingworth}, \cite{Stukow1}, \cite{Stukow2}]\label{thm_Stukow}
For $(g,n)=(2,0)$, $(2,1)$ and $(3,0)$, we have the following presentation for $\mathcal{M}(N_{g,n})$:
\begin{eqnarray*}
\mathcal{M}(N_2)&=&\bigl< a_1, y \mid a_1^2=y^2=(a_1y)^2=1\bigr> \cong \mathbb Z_2\oplus \mathbb Z_2,\\
\mathcal{M}(N_{2,1})&=&\bigl< a_1, y \mid ya_1y^{-1}=a_1^{-1}\bigr>,\\ 
\mathcal{M}(N_3)&=&\bigl< a_1, a_2, y \mid a_1a_2a_1=a_2a_1a_2, y^2=(a_1y)^2=(a_2y)^2=(a_1a_2)^6=1 \bigr>.
\end{eqnarray*}

If $g\geq 4$ and $n\in \{ 0,1\}$ or $(g,n)=(3,1)$, then $\mathcal{M}(N_{g,n})$ admits a presentation with generators $a_1,\dots , a_{g-1}, y$, and $b$ for $g\geq 4$. The defining relations are
\begin{enumerate}
 \item[(A1)] $[a_i,a_j]=1$\hspace{1cm} for $g\geq 4$, $|i-j|>1$,
 \item[(A2)] $a_ia_{i+1}a_i=a_{i+1}a_ia_{i+1}$\hspace{1cm} for $i=1,\dots ,g-2$,
 \item[(A3)] $[a_i,b]=1$\hspace{1cm} for $g\geq 4$, $i\not=4$,
 \item[(A4)] $a_4ba_4=ba_4b$\hspace{1cm} for $g\geq 5$,
 \item[(A5)] $(a_2a_3a_4b)^{10}=(a_1a_2a_3a_4b)^6$\hspace{1cm} for $g\geq 5$,
 \item[(A6)] $(a_2a_3a_4a_5a_6b)^{12}=(a_1a_2a_3a_4a_5a_6b)^9$\hspace{1cm} for $g\geq 7$,
 \item[(A9a)] $[b_2,b]=1$\hspace{1.0cm} for $g=6$,
 \item[(A9b)] $[a_{g-5},b_{\frac{g-2}{2}}]=1$\hspace{1.0cm} for $g\geq 8$ even,\\
 where $b_0=a_1$, $b_1=b$ and\\ $b_{i+1}=(b_{i-1}a_{2i}a_{2i+1}a_{2i+2}a_{2i+3}b_i)^5(b_{i-1}a_{2i}a_{2i+1}a_{2i+2}a_{2i+3})^{-6}$ \\
for $1\leq i\leq \frac{g-4}{2}$,
 \item[(B1)] $y(a_2a_3a_1a_2ya_2^{-1} a_1^{-1}a_3^{-1}a_2^{-1}) = (a_2a_3a_1a_2ya_2^{-1}a_1^{-1}a_3^{-1}a_2^{-1})y$ \hspace{1cm}for $g\geq 4$,
 \item[(B2)] $y(a_2a_1y^{-1}a_2^{-1}ya_1a_2)y=a_1(a_2a_1y^{-1}a_2^{-1}ya_1a_2)a_1$,
 \item[(B3)] $[a_i,y]=1$\hspace{1cm} for $g\geq 4$, $i=3,\dots ,g-1$,
 \item[(B4)] $a_2(ya_2y^{-1}) = (ya_2y^{-1})a_2$,
 \item[(B5)] $ya_1=a_1^{-1}y$,
 \item[(B6)] $byby^{-1} = \{a_1a_2a_3(y^{-1}a_2y)a_3^{-1}a_2^{-1}a_1^{-1} \}\{a_2^{-1}a_3^{-1}(ya_2y^{-1})a_3a_2\}$\hspace{1cm} for $g\geq 4$,
 \item[(B7)] $[(a_4a_5a_3a_4a_2a_3a_1a_2ya_2^{-1}a_1^{-1}a_3^{-1}a_2^{-1}a_4^{-1}a_3^{-1}a_5^{-1}a_4^{-1}),b] =1$\hspace{1.0cm} for $g\geq 6$,
 \item[(B8)] $\{(ya_1^{-1}a_2^{-1}a_3^{-1}a_4^{-1})b(a_4a_3a_2a_1y^{-1})\}\{(a_1^{-1}a_2^{-1}a_3^{-1}a_4^{-1})b^{-1}(a_4a_3a_2a_1)\}$
	     $=\{(a_4^{-1}a_3^{-1}a_2^{-1})y(a_2a_3a_4)\}\{a_3^{-1}a_2^{-1}y^{-1}a_2a_3\}\{a_2^{-1}ya_2\}y^{-1}$\hspace{0.9cm} for $g\geq 5$,
 \item[(C1)] $(a_1a_2\cdots a_{g-1})^g=1$\hspace{1.0cm} for $g\geq 4$ even and $n=0$,
 \item[(C2)] $[a_1,\rho ]=1$ \hspace{1.0cm}for $g\geq 4$ and $n=0$,\\
 where $\rho =(a_1a_2\cdots a_{g-1})^g$ for $g$ odd and\\
 $\rho =(y^{-1}a_2a_3\cdots a_{g-1}ya_2a_3\cdots a_{g-1})^{\frac{g-2}{2}}y^{-1}a_2a_3\cdots a_{g-1}$ for $g$ even,
 \item[(C3)] $\rho ^2=1$\hspace{1.0cm}for $g\geq 4$ and $n=0$,
 \item[(C4)] $(y^{-1}a_2a_3\cdots a_{g-1}ya_2a_3\cdots a_{g-1})^{\frac{g-1}{2}}=1$\hspace{1.0cm} for $g\geq 4$ odd and $n=0$,
\end{enumerate}
where $[x_1,x_2]=x_1x_2x_1^{-1}x_2^{-1}$.
\end{thm}

\section{Presentation for $\mathcal{M}(N_{g,n})$}\label{section-mainthm}

The main theorem in this paper is as follows:

\begin{thm}\label{main-thm}
For $g\geq 1$ and $n\in \{ 0,1\}$, $\mathcal{M}(N_{g,n})$ has the following presentation:

generators: $\{ t_c \mid c:\hspace{0.05cm} \text{two-sided s.c.c. on }N_{g,n} \}$\\
\hspace{2.01cm}$\cup \{ Y_{\mu ,\alpha } \mid \mu :\hspace{0.05cm} \text{one-sided s.c.c. on } N_{g,n},\ \alpha :\hspace{0.05cm} \text{s.c.c. on } N_{g,n},\ |\mu \cap \alpha |=1\}$.\\
Denote the generating set by $X$. 

relations:
\begin{enumerate}
 \item[(0)] $t_c=1$ when $c$ bounds a disk or a M\"{o}bius band in $N_{g,n}$,
 \item[(I)] All the braid relations\\
 \[
\left\{ \begin{array}{lll}
(i) & ft_cf^{-1}=t_{f(c)}^{\varepsilon _{f(c)}}&\text{for} \ f\in X,   \\
(ii) & fY_{\mu ,\alpha }f^{-1}=Y_{f(\mu ),f(\alpha )}^{\varepsilon _{f(\alpha )}}&\text{for} \ f\in X,
 \end{array} \right.
 \]
 \item[(I\hspace{-0.06cm}I)] All the 2-chain relations,
 \item[(I\hspace{-0.06cm}I\hspace{-0.06cm}I)] All the lantern relations,
 \item[(I\hspace{-0.04cm}V)] All the relations in Lemma~\ref{pushing3},
 i.e. $Y_{\mu ,\alpha \beta }=Y_{\mu ,\alpha }Y_{\mu ,\beta }$,
 \item[(V)] All the relations in Lemma~\ref{pushing2},
 i.e. $Y_{\mu ,\alpha }=t_{\delta _1}^{\varepsilon _{\delta _1}}t_{\delta _2}^{\varepsilon _{\delta _2}}$.
\end{enumerate} 
\end{thm}

In (I) and (I\hspace{-0.04cm}V) one can substitute the right hand side of (V) for each generator $Y_{\mu ,\alpha }$ with one-sided $\alpha $. Then one can remove the generators $Y_{\mu ,\alpha }$ with one-sided $\alpha $ and relations~(V) from the presentation.


We denote by $G$ the group which has the presentation in Theorem~\ref{main-thm}. Let $\iota :\Sigma _{h,m}\hookrightarrow N_{g,n}$ be a smooth embedding and let $G^\prime $ be the group whose presentation has all Dehn twists along simple closed curves on $\Sigma _{h,m}$ as generators and Relations~(0$^\prime $), (I$^\prime $), (I\hspace{-0.06cm}I) and (I\hspace{-0.06cm}I\hspace{-0.06cm}I) in Theorem~\ref{pres_Gervais}. By Theorem~\ref{pres_Gervais}, $\mathcal{M}(\Sigma _{h,m})$ is isomorphic to $G^\prime $, and we have the homomorphism $G^\prime \rightarrow G$ defined by the correspondence of $t_c$ to $t_{\iota (c)}^{\varepsilon _{\iota (c)}}$, where $\varepsilon _{\iota (c)}=1$ if the restriction $\iota |_{\mathcal{N}_{\Sigma _{h,m}}(c)}:\mathcal{N}_{\Sigma _{h,m}}(c)\rightarrow \mathcal{N}_{N_{g,n}}(\iota (c))$ is orientation preserving, and $\varepsilon _{\iota (c)}=-1$ if the restriction $\iota |_{\mathcal{N}_{\Sigma _{h,m}}(c)}:\mathcal{N}_{\Sigma _{h,m}}(c)\rightarrow \mathcal{N}_{N_{g,n}}(\iota (c))$ is orientation reversing. Then we remark the following.  

\begin{rem}\label{orientable}
The composition $\iota _\ast :\mathcal{M}(\Sigma _{h,m})\rightarrow G$ of the isomorphism $\mathcal{M}(\Sigma _{h,m})\rightarrow G^\prime $ and the homomorphism $G^\prime \rightarrow G$ is a homomorphism. In particular, if a product $t_{c_1}^{\varepsilon _1}t_{c_2}^{\varepsilon _2}\cdots t_{c_k}^{\varepsilon _k}$ of Dehn twists along simple closed curves $c_1$, $c_2$, $\dots $, $c_k$ on a connected compact orientable subsurface of $N_{g,n}$ is equal to the identity map in the mapping class group of the subsurface, then $t_{c_1}^{\varepsilon _1}t_{c_2}^{\varepsilon _2}\cdots t_{c_k}^{\varepsilon _k}$ is equal to $1$ in $G$. That means such a relation $t_{c_1}^{\varepsilon _1}t_{c_2}^{\varepsilon _2}\cdots t_{c_k}^{\varepsilon _k}=1$ is obtained from Relations~(0), (I), (I\hspace{-0.06cm}I) and (I\hspace{-0.06cm}I\hspace{-0.06cm}I).
\end{rem}


Set $X^\pm :=X\cup \{ x^{-1} \mid x\in X \}$. By Relation~(I), we have the following lemma.

\begin{lem}\label{braid3}
For $f\in G$, suppose that $f=f_1f_2\dots f_k$, where $f_1$, $f_2$, $\dots $, $f_k\in X^\pm $. Then we have
\[
\left\{ \begin{array}{ll}
(i) & ft_cf^{-1}=t_{f(c)}^{\varepsilon _{f(c)}},   \\
(ii) & fY_{\mu ,\alpha }f^{-1}=Y_{f(\mu ),f(\alpha )}^{\varepsilon _{f(\alpha )}}. \end{array} \right.
\]
\end{lem}

The next lemma follows from an argument of the combinatorial group theory (for instance, see \cite[Lemma~4.2.1, p42]{Johnson}).

\begin{lem}\label{combinatorial}
For groups $\Gamma $, $\Gamma ^\prime $ and $F$, a surjective homomorphism $\pi :F\rightarrow \Gamma $ and a homomorphism $\nu :F\rightarrow \Gamma ^\prime $, we define a map $\nu ^\prime :\Gamma \rightarrow \Gamma ^\prime $ by $\nu ^\prime (x):=\nu (\widetilde{x})$ for $x\in \Gamma $, where $\widetilde{x}\in F$ is a lift of $x$ with respect to $\pi $ (see the diagram below).

Then if ${\rm ker}\pi \subset {\rm ker}\nu $, $\nu ^\prime $ is well-defined and a homomorphism.
\end{lem}

\[
\xymatrix{
 F \ar@{->>}[d]_\pi \ar[dr]^{\nu } &  \\
\Gamma  \ar@{-->}[r]_{\nu ^\prime}  & \Gamma ^\prime  \\
}
\]

\begin{proof}[Proof of Theorem~\ref{main-thm}]
$\mathcal{M}(N_1)$ and $\mathcal{M}$ are trivial (see~\cite{Epstein}). Assume $g\geq 2$ and $n\in \{ 0,1\}$. Then we obtain Theorem~\ref{main-thm} if $\mathcal{M}(N_{g,n})$ is isomorphic to $G$. Let $\varphi :G\rightarrow \mathcal{M}(N_{g,n})$ be the surjective homomorphism defined by $\varphi (t_c):=t_c$ and $\varphi (Y_{\mu ,\alpha }):=Y_{\mu ,\alpha }$. 

Set $X_0:=\{ a_1,\dots ,a_{g-1},b , y\}\subset \mathcal{M}(N_{g,n})$ for $g\geq 4$ and $X_0:=\{ a_1,\dots ,a_{g-1}, y\}\subset \mathcal{M}(N_{g,n})$ for $g=2$, $3$. Let $F(X_0)$ be the free group which is freely generated by $X_0$ and let $\pi :F(X_0)\rightarrow \mathcal{M}(N_{g,n})$ be the natural projection (by Theorem~\ref{thm_Stukow}). We define the homomorphism $\nu :F(X_0)\rightarrow G$ by $\nu (a_i):=a_i$ for $i=1$, $\dots $, $g-1$, $\nu (b):=b$ and $\nu (y):=y$, and a map $\psi =\nu ^\prime :\mathcal{M}(N_{g,n})\rightarrow G$ by $\psi (a_i^{\pm 1}):=a_i^{\pm 1}$ for $i=1$, $\dots $, $g-1$, $\psi (b^{\pm 1}):=b^{\pm 1}$, $\psi (y^{\pm 1}):=y^{\pm 1}$ and $\psi (f):=\nu (\widetilde{f})$ for the other $f\in \mathcal{M}(N_{g,n})$, where $\widetilde{f}\in F(X_0)$ is a lift of $f$ with respect to $\pi $ (see the diagram below).

\[
\xymatrix{
 F(X_0) \ar@{->>}[d]_\pi \ar[dr]^{\nu } &  \\
\mathcal{M}(N_{g,n}) \ar@{-->}[r]_{\psi }  & G  \\
}
\]

If $\psi $ is a homomorphism, $\varphi \circ \psi ={\rm id}_{\mathcal{M}(N_{g,n})}$ by the definition of $\varphi $ and $\psi $. Thus it is sufficient for proving that $\psi $ is isomorphism to show that $\psi $ is a homomorphism and surjective.

\subsection{Proof that $\psi $ is a homomorphism}

$\mathcal{M}(N_1)$ and $\mathcal{M}(N_{1,1})$ are trivial (see \cite[Theorem~3.4]{Epstein}). For $(g,n)\in \{ (2,0),(2,1),(3,0)\}$, relations of the presentation in Theorem~\ref{thm_Stukow} are obtained from Relations~(0), (I), (I\hspace{-0.06cm}I), (I\hspace{-0.06cm}I\hspace{-0.06cm}I), (I\hspace{-0.04cm}V) and (V), clearly. Thus by Lemma~\ref{combinatorial}, $\psi $ is a homomorphism.

Assume $g\geq 4$ or $(g,n)=(3,1)$. By Lemma~\ref{combinatorial}, if the relations of the presentation in Theorem~\ref{thm_Stukow} are obtained from Relations~(0), (I), (I\hspace{-0.06cm}I), (I\hspace{-0.06cm}I\hspace{-0.06cm}I), (I\hspace{-0.04cm}V) and (V), then $\psi $ is a homomorphism. 

The group generated by $a_1$, $\dots $, $a_{g-1}$ and $b$ with Relations~(A1)-(A9b) as defining relations is isomorphic to $\mathcal{M}(\Sigma _{h,1})$ (resp. $\mathcal{M}(\Sigma _{h,2})$) for $g=2h+1$ (resp. $g=2h+2$) by Theorem~3.1 in \cite{Paris-Szepietowski}, and Relations~(A1)-(A9b) are relations on the mapping class group of the orientable subsurface $\mathcal{N}_{N_{g,n}}(\alpha _1\cup \cdots \cup \alpha _{g-1})$ of $N_{g,n}$. Hence Relations~(A1)-(A9b) are obtained from Relations~(0), (I), (I\hspace{-0.06cm}I) and (I\hspace{-0.06cm}I\hspace{-0.06cm}I) by Remark~\ref{orientable}.

Stukow~\cite{Stukow2} gave geometric interpretations for Relations~(B1)-(B8) in Section~4 in \cite{Stukow2}. By the interpretation, Relations~(B1), (B2), (B3), (B4), (B5), (B7) are obtained from Relations~(I) (use Lemma~\ref{braid3}), Relation~(B6) is obtained from Relations~(0), (I), (I\hspace{-0.06cm}I\hspace{-0.06cm}I), (I\hspace{-0.04cm}V) and (V) (use Lemma~\ref{pushing4} and Lemma~\ref{braid3}), and Relation~(B8) is obtained from Relations~(I), (I\hspace{-0.04cm}V) and (V) (use Lemma~\ref{braid3}). Thus $\psi $ is a homomorphism when $n=1$.

We assume $n=0$. By Remark~\ref{orientable}, $k$-chain relations are obtained from Relations~(0), (I), (I\hspace{-0.06cm}I) and (I\hspace{-0.06cm}I\hspace{-0.06cm}I) for each $k$. Relation~(C1) is interpreted in $G$ as follows.
\begin{eqnarray*}
(a_1a_2\cdots a_{g-1})^g \stackrel{\text{(0),(I),(I\hspace{-0.06cm}I),(I\hspace{-0.06cm}I\hspace{-0.06cm}I)}}{=}t_{\gamma _{1,2,\dots ,g}}t_{\gamma _{1,2,\dots ,g}}^{-1}=1.
\end{eqnarray*}
Thus Relation~(C1) is  obtained from Relations~(0), (I), (I\hspace{-0.06cm}I) and (I\hspace{-0.06cm}I\hspace{-0.06cm}I).


Relation~(C2) is obtained from Relations~(I) by Lemma~\ref{braid3}, clearly.

When $g$ is odd, by using the $(g-1)$-chain relation, Relation~(C3) is interpreted in $G$ as follows.
\begin{eqnarray*}
\rho ^2=(a_1a_2\cdots a_{g-1})^{2g} \stackrel{\text{(0),(I),(I\hspace{-0.06cm}I),(I\hspace{-0.06cm}I\hspace{-0.06cm}I)}}{=}t_{\partial \mathcal{N}_{N_g}(\gamma _{1,2,\dots ,g})}^\varepsilon \stackrel{(0)}{=}1,
\end{eqnarray*}
where $\varepsilon $ is $1$ or $-1$. Note that $\mathcal{N}_{N_g}(\gamma _{1,2,\dots ,g})$ is a M\"{o}bius band in $N_g$. Thus Relation~(C3) is  obtained from Relations~(0), (I), (I\hspace{-0.06cm}I) and (I\hspace{-0.06cm}I\hspace{-0.06cm}I) when $g$ is odd.

When $g$ is even, we rewrite the left-hand side $\rho ^2$ of Relation~(C3) by braid relations. Set $A:=a_2a_3\cdots a_{g-1}$. Note that 
\[
Y_{\mu _1,\gamma _{1,2,3}}A^2Y_{\mu _1,\gamma _{1,2,\dots ,2i-1}}A^{-2}=Y_{\mu _1,\gamma _{1,2,\dots ,2i+1}}
\]
for $i=2$, $\dots $, $\frac{g-2}{2}$  by Relation~(I), (I\hspace{-0.06cm}V) and then we have
\begin{eqnarray*}
& &\rho \\
&=& y^{-1}A(yAy^{-1}A)^{\frac{g-2}{2}}\\
&\stackrel{\text{(I)}}{=}& y^{-1}A(ya_2y^{-1}a_3\cdots a_{g-1}A)^{\frac{g-2}{2}}\\
&=& y^{-1}A(\underline{y(a_2y^{-1}a_2^{-1})}A^2)^{\frac{g-2}{2}}\\
&\stackrel{(\text{I}),(\text{I\hspace{-0.06cm}V})}{=}& y^{-1}A(Y_{\mu _1, \gamma _{1,2,3}} A^2)^{\frac{g-2}{2}}.\\
&=& y^{-1}AY_{\mu _1, \gamma _{1,2,3}}A^2\cdots Y_{\mu _1, \gamma _{1,2,3}}A^2Y_{\mu _1, \gamma _{1,2,3}}A^2Y_{\mu _1, \gamma _{1,2,3}}A^2\\
&=& y^{-1}AY_{\mu _1, \gamma _{1,2,3}}A^2\cdots Y_{\mu _1, \gamma _{1,2,3}}A^2\underline{Y_{\mu _1, \gamma _{1,2,3}}A^2Y_{\mu _1, \gamma _{1,2,3}}A^{-2}}A^4\\
&\stackrel{(\text{I}),(\text{I\hspace{-0.06cm}V})}{=}& y^{-1}AY_{\mu _1, \gamma _{1,2,3}}A^2\cdots Y_{\mu _1, \gamma _{1,2,3}}A^2Y_{\mu ,\gamma _{1,2,3,4,5}}A^4\\
&=& y^{-1}AY_{\mu _1, \gamma _{1,2,3}}A^2\cdots \underline{Y_{\mu _1, \gamma _{1,2,3}}A^2Y_{\mu ,\gamma _{1,2,3,4,5}}A^{-2}}A^6\\
&\stackrel{(\text{I}),(\text{I\hspace{-0.06cm}V})}{=}& y^{-1}AY_{\mu _1, \gamma _{1,2,3}}A^2\cdots Y_{\mu _1, \gamma _{1,2,3,4,5,6,7}}A^6\\
&\vdots & \\
&\stackrel{(\text{I}),(\text{I\hspace{-0.06cm}V})}{=}& y^{-1}AY_{\mu _1,\gamma _{1,2,\dots ,g-1}}A^{g-2}\\
&=& \underline{y^{-1}\cdot AY_{\mu _1,\gamma _{1,2,\dots ,g-1}}A^{-1}}\cdot A^{g-1}\\
&\stackrel{(\text{I}),(\text{I\hspace{-0.06cm}V})}{=}& Y_{\mu _1,\gamma _{1,2,\dots ,g}}A^{g-1}.
\end{eqnarray*}
Since $Y_{\mu _1,\gamma _{1,2,\dots ,g}}$ commutes with $a_i$ for $i=2$, $\dots $, $g-1$, and $\partial \mathcal{N}_{N_g}(\mu _1\cup \gamma _{1,2,\dots ,g})=\partial \mathcal{N}_{N_g}(\alpha _2\cup \dots \cup \alpha _{g-1})$ (see Figure~\ref{chain1_nonori}), we have
\begin{eqnarray*}
\rho ^2 &=& Y_{\mu _1,\gamma _{1,2,\dots ,g}}A^{g-1}Y_{\mu _1,\gamma _{1,2,\dots ,g}}A^{g-1}\\
&\stackrel{(\text{I})}{=}& Y_{\mu _1,\gamma _{1,2,\dots ,g}}^2A^{2g-2}\\
&\stackrel{(\text{0}),(\text{I}),(\text{I\hspace{-0.06cm}I}),(\text{I\hspace{-0.06cm}I\hspace{-0.06cm}I})}{=}& Y_{\mu _1,\gamma _{1,2,\dots ,g}}^2t_{\partial \mathcal{N}_{N_g}(\alpha _2\cup \dots \cup \alpha _{g-1})} \\
&\stackrel{\text{Lem.~\ref{pushing4}}}{=}&t_{\partial \mathcal{N}_{N_g}(\alpha _2\cup \dots \cup \alpha _{g-1})}^{-1}t_{\partial \mathcal{N}_{N_g}(\alpha _2\cup \dots \cup \alpha _{g-1})} \\
&=& 1.
\end{eqnarray*}
Recall that the relations in Lemma~\ref{pushing4} are obtained from Relations~(0), (I\hspace{-0.04cm}V) and (V). Thus Relation~(C3) is  obtained from Relations~(0), (I), (I\hspace{-0.06cm}I), (I\hspace{-0.06cm}I\hspace{-0.06cm}I), (I\hspace{-0.04cm}V) and (V) when $g$ is even.

\begin{figure}[h]
\includegraphics[scale=0.7]{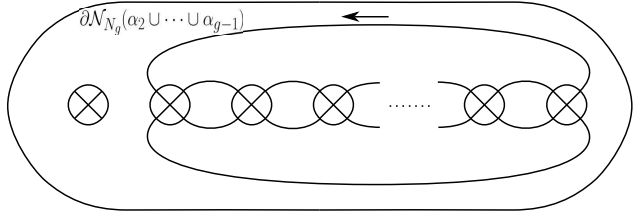}
\caption{Simple closed curve $\partial \mathcal{N}_{N_g}(\alpha _2\cup \dots \cup \alpha _{g-1})$ on $N_g$.}\label{chain1_nonori}
\end{figure}

Finally, we also rewrite the left-hand side $(y^{-1}a_2a_3\cdots a_{g-1}ya_2a_3\cdots a_{g-1})^{\frac{g-1}{2}}$ of Relation~(C4) by braid relations. Remark that $g$ is odd. For $1\leq i_1<i_2<\dots <i_k \leq g$, we denote by $\gamma _{i_1,i_2,\dots ,i_k}^\prime $ the simple closed curve on $N_{g,n}$ as in Figure~\ref{scc_gammaprime}. Note that 
\[
Y_{\mu _1,\gamma ^\prime _{1,2,3}}A^2Y_{\mu _1,\gamma ^\prime _{1,2,\dots ,2i-1}}A^{-2}=Y_{\mu _1,\gamma ^\prime _{1,2,\dots ,2i+1}}
\]
for $i=2$, $\dots $, $\frac{g-1}{2}$ and $\partial \mathcal{N}_{N_g}(\mu _1\cup \gamma _{1,2,\dots ,g})=d_1\sqcup d_2$ as in Figure~\ref{chain2_nonori}, and by a similar argument as for Relation~(C3) when $g$ is even, we have
\begin{eqnarray*}
& & (y^{-1}a_2a_3\cdots a_{g-1}ya_2a_3\cdots a_{g-1})^{\frac{g-1}{2}}\\
&=& (y^{-1}AyA)^{\frac{g-1}{2}}\\
&\stackrel{(\text{I})}{=}& (\underline{y^{-1}(a_2ya_2^{-1})}A^2)^{\frac{g-1}{2}}\\
&\stackrel{(\text{I}),(\text{I\hspace{-0.06cm}V})}{=}& (Y_{\mu _1,\gamma ^\prime _{1,2,3}}A^2)^{\frac{g-1}{2}}\\
&=& Y_{\mu _1,\gamma ^\prime _{1,2,3}}A^2\cdots Y_{\mu _1,\gamma ^\prime _{1,2,3}}A^2Y_{\mu _1,\gamma ^\prime _{1,2,3}}A^2Y_{\mu _1,\gamma ^\prime _{1,2,3}}A^2\\
&=& Y_{\mu _1,\gamma ^\prime _{1,2,3}}A^2\cdots Y_{\mu _1,\gamma ^\prime _{1,2,3}}A^2\underline{Y_{\mu _1,\gamma ^\prime _{1,2,3}}A^2Y_{\mu _1,\gamma ^\prime _{1,2,3}}A^{-2}}A^4\\
&\stackrel{(\text{I}),(\text{I\hspace{-0.06cm}V})}{=}& Y_{\mu _1,\gamma ^\prime _{1,2,3}}A^2\cdots Y_{\mu _1,\gamma ^\prime _{1,2,3}}A^2Y_{\mu _1,\gamma ^\prime _{1,2,3,4,5}}A^4\\
&=& Y_{\mu _1,\gamma ^\prime _{1,2,3}}A^2\cdots \underline{Y_{\mu _1,\gamma ^\prime _{1,2,3}}A^2Y_{\mu _1,\gamma ^\prime _{1,2,3,4,5}}A^{-2}}A^6\\
&\stackrel{(\text{I}),(\text{I\hspace{-0.06cm}V})}{=}& Y_{\mu _1,\gamma ^\prime _{1,2,3}}A^2\cdots Y_{\mu ,\gamma ^\prime _{1,2,3,4,5,6,7}}A^6\\
&\vdots &\\
&\stackrel{(\text{I}),(\text{I\hspace{-0.06cm}V})}{=}& Y_{\mu ,\gamma _{1,2,\dots ,g}}A^{g-1}\\
&\stackrel{\text{(0),(I),(I\hspace{-0.06cm}I),(I\hspace{-0.06cm}I\hspace{-0.06cm}I)}}{=}& Y_{\mu ,\gamma _{1,2,\dots ,g}}t_{d_1}t_{d_2}\\
&\stackrel{(\text{V})}{=}& t_{d_1}^{-1}t_{d_2}^{-1}t_{d_1}t_{d_2}\\
&\stackrel{(\text{I})}{=}& 1,
\end{eqnarray*}
where simple closed curves $d_1$ and $d_2$ are boundary components of $\mathcal{N}_{N_g}(\alpha _2\cup \cdots \cup \alpha _{g-1})$ as in Figure~\ref{chain2_nonori}. Therefore Relation~(C4) is obtained from Relations~(I), (I\hspace{-0.06cm}I), (I\hspace{-0.06cm}V) and (V), and $\psi :\mathcal{M}(N_{g,n})\rightarrow G$ is a homomorphism.

\begin{figure}[h]
\includegraphics[scale=0.7]{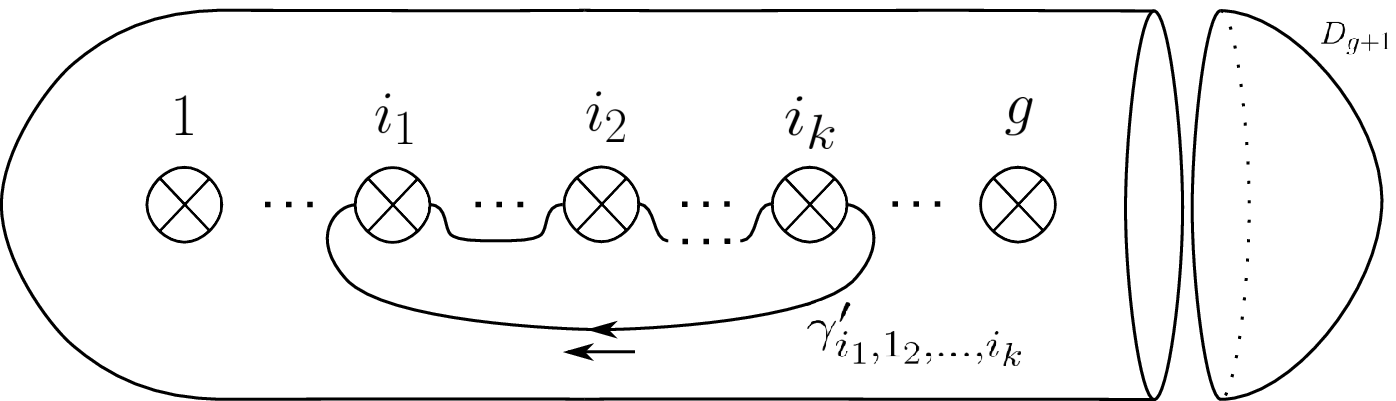}
\caption{Simple closed curve $\gamma ^\prime _{i_1,i_2,\dots ,i_k}$ on $N_{g,n}$.}\label{scc_gammaprime}
\end{figure}

\begin{figure}[h]
\includegraphics[scale=0.7]{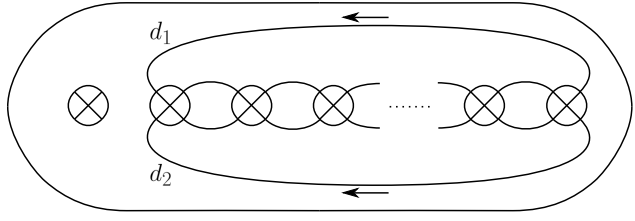}
\caption{Simple closed curve $d_1$ and $d_2$ on $N_{g,n}$.}\label{chain2_nonori}
\end{figure}

\subsection{Surjectivity of $\psi $}

We show that there exist lifts of $t_c$'s and $Y_{\mu ,\alpha }$'s with respect to $\psi $ for cases below, to prove the surjectivity of $\psi $.
\begin{itemize}
\item[(1)] $t_c$; $c$ is non-separating and $N_{g,n}-c$ is non-orientable,
\item[(2)] $t_c$; $c$ is non-separating and $N_{g,n}-c$ is orientable,
\item[(3)] $t_c$; $c$ is separating,
\item[(4)] $Y_{\mu ,\alpha }$; $\alpha $ is two-sided and $N_{g,n}-\alpha $ is non-orientable,
\item[(5)] $Y_{\mu ,\alpha }$; $\alpha $ is two-sided and $N_{g,n}-\alpha $ is orientable,
\item[(6)] $Y_{\mu ,\alpha }$; $\alpha $ is one-sided.
\end{itemize}

Set $X_0^\pm :=X_0\cup \{ x^{-1} \mid x\in X_0 \}$, and for a simple closed curve $c$ on $N_{g,n}$, we denote by $(N_{g,n})_c$ the surface obtained from $N_{g,n}$ by cutting $N_{g,n}$ along $c$. 

\textbf{Case~(1).} Since $(N_{g,n})_c$ is diffeomorphic to $N_{g-2,n+2}$ and $g\geq 3$, there exists a product $f=f_1f_2\cdots f_k\in \mathcal{M}(N_{g,n})$ of $f_1$, $f_2$, $\cdots $, $f_k \in X_0^\pm $ such that $f(\alpha _1)=c$. Note that $\psi (f_i)=f_i \in X^\pm \subset G$ for $i=1$, $2$, $\dots $, $k$. Thus we have 
\begin{eqnarray*}
\psi (fa_1f^{-1}) &=& \psi (f)\psi (a_1)\psi (f)^{-1}\\
&=& f_1f_2\cdots f_ka_1f_k^{-1}\cdots f_2^{-1}f_1^{-1}\\
&\stackrel{\text{Lem.~\ref{braid3}}}{=}& t_{f(\alpha _1)}^\varepsilon \\
&=& t_c^\varepsilon ,
\end{eqnarray*}
where $\varepsilon $ is $1$ or $-1$. Thus $fa_1^\varepsilon f^{-1}\in \mathcal{M}(N_{g,n})$ is a lift of $t_c\in G$ with respect to $\psi $ for some $\varepsilon \in \{ -1,1\}$.

\textbf{Case~(2).} We remark that $g$ is even in this case. When $g=2$, such a simple closed curve $c$ is unique and $c=\alpha _1$. Thus $a_1\in \mathcal{M}(N_{g,n})$ is the lift of $t_c\in G$ with respect to $\psi $. When $g=4$, since $(N_{g,n})_c$ is diffeomorphic to $\Sigma _{1,n+2}$, there exists a product $f=f_1f_2\cdots f_k\in \mathcal{M}(N_{g,n})$ of $f_1$, $f_2$, $\cdots $, $f_k \in X_0^\pm $ such that $f(\beta )=c$. By a similar argument as in Case~(1), $fb^\varepsilon f^{-1}\in \mathcal{M}(N_{g,n})$ is a lift of $t_c\in G$ with respect to $\psi $ for some $\varepsilon \in \{ -1,1\}$.

Assume $g\geq 6$ even. Then there exists a product $f=f_1f_2\cdots f_k\in \mathcal{M}(N_{g,n})$ of $f_1$, $f_2$, $\cdots $, $f_k \in X_0^\pm $ such that $f(\gamma _{1,2,\dots ,g})=c$. Since $\alpha _1\cup \alpha _3\cup \gamma _{5,6,\dots ,g}\cup \gamma _{1,2,\dots ,g}$ bounds a subsurface of $N_{g,n}$ which is diffeomorphic to $\Sigma _{0,4}$ (see Figure~\ref{lantern2}), we have $bt_{\gamma _{3,4,\dots ,g}}t_{\gamma _{1,2,5,\dots ,g}}=t_{\gamma _{1,2,\dots ,g}}a_1a_3t_{\gamma _{5,6,\dots ,g}}$ by a lantern relation. Note that $b$, $t_{\gamma _{3,4,\dots ,g}}$, $t_{\gamma _{1,2,5,\dots ,g}}$, $a_1$, $a_3$, $t_{\gamma _{5,6,\dots ,g}}$ are Dehn twists of type~(1), and $t_{\gamma _{3,4,\dots ,g}}$, $t_{\gamma _{1,2,5,\dots ,g}}$, $t_{\gamma _{5,6,\dots ,g}}\in G$ have lifts $h_1$, $h_2$, $h_3\in \mathcal{M}(N_{g,n})$ with respect to $\psi $, respectively. Thus we have
\begin{eqnarray*}
& & \psi (fbh_1h_2a_1^{-1}a_3^{-1}h_3^{-1}f^{-1})\\
&=& f_1 f_2\cdots f_kbt_{\gamma _{3,4,\dots ,g}}t_{\gamma _{1,2,5,\dots ,g}}a_1^{-1}a_3^{-1}t_{\gamma _{5,6,\dots ,g}}^{-1}f_k^{-1}\cdots f_2^{-1}f_1^{-1}\\
&\stackrel{(\text{I\hspace{-0.06cm}I\hspace{-0.06cm}I})}{=}& f_1f_2\cdots f_kt_{\gamma _{1,2,\dots ,g}}f_k^{-1}\cdots f_2^{-1}f_1^{-1}\\
&\stackrel{\text{Lem.~\ref{braid3}}}{=}& t_c^\varepsilon ,
\end{eqnarray*}
where $\varepsilon $ is $1$ or $-1$. Thus $f(bh_1h_2a_1^{-1}a_3^{-1}h_3^{-1})^\varepsilon f^{-1}\in \mathcal{M}(N_{g,n})$ is a lift of $t_c\in G$ with respect to $\psi $ for some $\{ -1,1\}$.

\begin{figure}[h]
\includegraphics[scale=0.7]{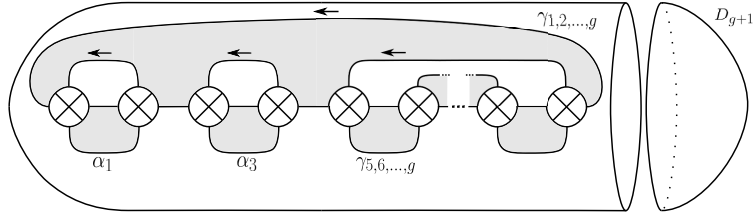}
\caption{$\alpha _1\cup \alpha _3\cup \gamma _{5,6,\dots ,g}\cup \gamma _{1,2,\dots ,g}$ bound a subsurface of $N_{g,n}$ which is diffeomorphic to $\Sigma _{0,4}$.}\label{lantern2}
\end{figure}

\textbf{Case~(4).} Since $N_{g,n}-{\rm int}\mathcal{N}_{N_{g,n}}(\mu \cup \alpha )$ is diffeomorphic to $N_{g-2,n+1}$ and the two-sided simple closed curve on $N_{2,1}$ is unique, there exists a product $f=f_1f_2\cdots f_k\in \mathcal{M}(N_{g,n})$ of $f_1$, $f_2$, $\cdots $, $f_k \in X_0^\pm $ such that $f(\alpha _1)=\alpha $ and $f(\mu _1)=\mu $. Thus we have
\begin{eqnarray*}
\psi (fyf^{-1}) &=& f_1f_2\cdots f_kyf_k^{-1}\cdots f_2^{-1}f_1^{-1}\\
&\stackrel{\text{Lem.~\ref{braid3}}}{=}& Y_{\mu ,\alpha }^\varepsilon ,
\end{eqnarray*}
where $\varepsilon $ is $1$ or $-1$. Thus $fy^\varepsilon f^{-1}\in \mathcal{M}(N_{g,n})$ is a lift of $Y_{\mu ,\alpha }\in G$ with respect to $\psi $ for some $\varepsilon \in \{ -1,1\}$.

\textbf{Case~(5).} We remark that $g$ is even in this case. Since $N_{g,n}-{\rm int}\mathcal{N}_{N_{g,n}}(\mu \cup \alpha )$ is diffeomorphic to $\Sigma _{\frac{g-2}{2},n+1}$ and the two-sided simple closed curve on $N_{2,1}$ is unique, there exists a product $f=f_1f_2\cdots f_k\in \mathcal{M}(N_{g,n})$ of $f_1$, $f_2$, $\cdots $, $f_k \in X_0^\pm $ such that $f(\gamma _{1,2,\dots ,g})=\alpha $ and $f(\mu _1)=\mu $. Note that $Y_{\mu _1,\gamma _{1,2}}$, $Y_{\mu _1,\gamma _{1,3}}$, $\dots $, $Y_{\mu _1,\gamma _{1,g}}$ are Y-homeomorphisms of type~(4), and $Y_{\mu _1,\gamma _{1,3}}$, $Y_{\mu _1,\gamma _{1,4}}$, $\dots $, $Y_{\mu _1,\gamma _{1,g}}\in G$ have lifts  $h_3$, $h_4$, $\dots $, $h_g\in \mathcal{M}(N_{g,n})$ with respect to $\psi $, respectively. Thus we have
\begin{eqnarray*}
& & \psi (fh_g\dots h_4h_3yf^{-1})\\
&=& f_1f_2\cdots f_kY_{\mu _1,\gamma _{1,g}}\dots Y_{\mu _1,\gamma _{1,4}}Y_{\mu _1,\gamma _{1,3}}yf_k^{-1}\cdots f_2^{-1}f_1^{-1}\\
&\stackrel{(\text{I\hspace{-0.06cm}V})}{=}& f_1f_2\cdots f_kY_{\mu _1,\gamma _{1,2,\dots ,g}}f_k^{-1}\cdots f_2^{-1}f_1^{-1}\\
&\stackrel{\text{Lem.~\ref{braid3}}}{=}& Y_{\mu ,\alpha }^\varepsilon ,
\end{eqnarray*}
where $\varepsilon $ is $1$ or $-1$. Thus $f(h_g\dots h_4h_3y)^\varepsilon f^{-1}\in \mathcal{M}(N_{g,n})$ is a lift of $Y_{\mu ,\alpha }\in G$ with respect to $\psi $ for some $\varepsilon \in \{ -1,1\}$.

\textbf{Case~(3).} Let $\Sigma $ be the component of $(N_{g,n})_c$ which has one boundary component. When $\Sigma $ is orientable, there exists a $k$-chain $c_1$, $c_2$, $\dots $, $c_k$ on $N_{g,n}$ such that $\mathcal{N}_{N_{g,n}}(c_1\cup c_2\cup \cdots \cup c_k)=\Sigma $. By the chain relation, $(t_{c_1}^{\varepsilon _1}t_{c_2}^{\varepsilon _2}\cdots t_{c_k}^{\varepsilon _k})^{2k+2}=t_c$ for some $\varepsilon _1$, $\varepsilon _2$, $\dots $, $\varepsilon _k\in \{ -1,1\}$.  Note that $t_{c_1}$, $t_{c_2}$, $\dots $, $t_{c_k}$ are Dehn twists of type~(1) and $t_{c_1}$, $t_{c_2}$, $\dots $, $t_{c_k}\in G$ have lifts $h_1$, $h_2$, $\dots $, $h_k\in \mathcal{M}(N_{g,n})$ with respect to $\psi $, respectively. Thus we have
\begin{eqnarray*}
\psi ((h_1^{\varepsilon _1}h_2^{\varepsilon _2}\dots h_k^{\varepsilon _k})^{2k+2})&=& (t_{c_1}^{\varepsilon _1}t_{c_2}^{\varepsilon _2}\cdots t_{c_k}^{\varepsilon _k})^{2k+2}\\
&\stackrel{(0),(\text{I}),(\text{I\hspace{-0.06cm}I}),(\text{I\hspace{-0.06cm}I\hspace{-0.06cm}I})}{=}& t_c.
\end{eqnarray*}
Thus $(h_1^{\varepsilon _1}h_2^{\varepsilon _2}\dots h_k^{\varepsilon _k})^{2k+2}\in \mathcal{M}(N_{g,n})$ is a lift of $t_c\in G$ with respect to $\psi $.

When $\Sigma $ is non-orientable, we proceed by induction on the genus $g^\prime $ of $\Sigma $. For $g^\prime =1$, $t_c=1$ by Relation~(0). 
When $g^\prime =2$ and $N_{g,n}-\Sigma $ is non-orientable, there exists a product $f=f_1f_2\cdots f_k\in \mathcal{M}(N_{g,n})$ of $f_1$, $f_2$, $\cdots $, $f_k \in X_0^\pm $ such that $f(\partial \mathcal{N}_{N_{g,n}}(\mu _1\cup \alpha _1))=c$. Hence $fy^2f^{-1}=t_c^\varepsilon $ for some $\varepsilon \in \{ -1,1\}$. Then we have
\begin{eqnarray*}
\psi (fy^2f^{-1})&=& f_1f_2\cdots f_k\ y^2f_k^{-1}\cdots f_2^{-1}f_1^{-1}\\
&\stackrel{\text{Lem.}~\ref{pushing4}}{=}& f_1f_2\cdots f_kt_{\partial \mathcal{N}_{N_{g,n}}(\mu _1\cup \alpha _1)}^{\varepsilon ^\prime }f_k^{-1}\cdots f_2^{-1}f_1^{-1}\\
&\stackrel{\text{Lem.~\ref{braid3}}}{=}& t_c^{\varepsilon },
\end{eqnarray*}
where $\varepsilon ^\prime $ is $1$ or $-1$. Thus $fy^{2\varepsilon }f^{-1}\in \mathcal{M}(N_{g,n})$ is a lift of $t_c\in G$ with respect to $\psi $. When $g^\prime =2$ and $N_{g,n}-\Sigma $ is orientable, $g$ is even and there exists a product $f=f_1f_2\cdots f_k\in \mathcal{M}(N_{g,n})$ of $f_1$, $f_2$, $\cdots $, $f_k \in X_0^\pm $ such that $f(\partial \mathcal{N}_{N_{g,n}}(\mu _1\cup \gamma _{1,2,\dots ,g}))=c$. Hence $fY_{\mu _1,\gamma _{1,2,\dots ,g}}^2f^{-1}=t_c^\varepsilon $ for some $\varepsilon \in \{ -1,1\}$. Since $Y_{\mu _1,\gamma _{1,2,\dots ,g}}$ is a Y-homeomorphism of type~(5), there exists a lift $h\in \mathcal{M}(N_{g,n})$ of $Y_{\mu _1,\gamma _{1,2,\dots ,g}}\in G$ with respect to $\psi $. By a similar argument above, $fh^{2\varepsilon }f^{-1}\in \mathcal{M}(N_{g,n})$ is a lift of $t_c\in G$ with respect to $\psi $.

Suppose $g^\prime \geq 3$. We take a diffeomorphism $f:\Sigma \rightarrow N_{g^\prime ,1}$ and simple closed curves $c_1, c_2, \dots , c_6$ and $c^\prime :=\partial N_{g^\prime ,1}=f(c)$ on $N_{g^\prime ,1}$ as in Figure~\ref{lantern4}. Note that $c^\prime \cup c_4 \cup c_5 \cup c_6$ bounds a subsurface of $N_{g^\prime ,1}$ which is diffeomorphic to $\Sigma _{0,4}$ and we have $t_{f^{-1}(c_1)}^{\varepsilon _1}t_{f^{-1}(c_2)}^{\varepsilon _2}t_{f^{-1}(c_3)}^{\varepsilon _3}t_{f^{-1}(c_4)}^{\varepsilon _4}=t_c^\varepsilon \in G$ for some $\varepsilon _1,\dots , \varepsilon _4, \varepsilon \in \{ -1,1\}$ by Relations~(0) and (I\hspace{-0.06cm}I\hspace{-0.06cm}I). Since each $c_i$ for $i=1$, $2$, $\dots $, $6$ bounds a subsurface of $N_{g,n}$ which is diffeomorphic to a non-orientable surface of genus $g_i<g^\prime $ with one boundary component and the complement of the subsurface is non-orientable, each $f^{-1}(c_i)$ $(i=1,2,\dots ,6)$ satisfies the inductive assumption. Hence $t_{f^{-1}(c_1)}$, $t_{f^{-1}(c_2)}$, $t_{f^{-1}(c_3)}$, $t_{f^{-1}(c_4)}\in G$ have lifts $h_1$, $h_2$, $h_3$ and $h_4\in \mathcal{M}(N_{g,n})$ with respect to $\psi $, respectively. Thus we have 
\begin{eqnarray*}
\psi (h_1^{\varepsilon _1}h_2^{\varepsilon _2}h_3^{\varepsilon _3}h_4^{\varepsilon _4})&=& t_{f^{-1}(c_1)}^{\varepsilon _1}t_{f^{-1}(c_2)}^{\varepsilon _2}t_{f^{-1}(c_3)}^{\varepsilon _3}t_{f^{-1}(c_4)}^{\varepsilon _4}\\
&\stackrel{(0),(\text{I\hspace{-0.06cm}I\hspace{-0.06cm}I})}{=}& t_c^\varepsilon .\\
\end{eqnarray*}
Thus $h_1^{\varepsilon _1}h_2^{\varepsilon _2}h_3^{\varepsilon _3}h_4^{\varepsilon _4}\in \mathcal{M}(N_{g,n})$ is a lift of $t_c\in G$ with respect to $\psi $.

\begin{figure}[h]
\includegraphics[scale=0.7]{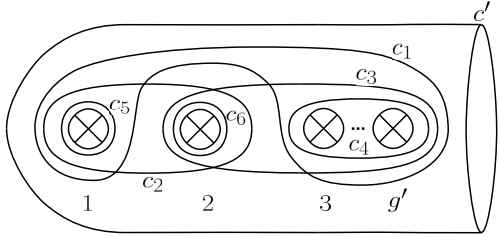}
\caption{Simple closed curves $c_1, c_2, \dots , c_6$ and $c^\prime $ on $N_{g^\prime ,1}$.}\label{lantern4}
\end{figure}


\textbf{Case~(6).} Let $\delta _1$, $\delta _2$ be two-sided simple closed curves on $N_{g,n}$ such that $\delta _1\sqcup \delta _2=\partial \mathcal{N}_{N_{g,n}}(\mu \cap \alpha )$. By Lemma~\ref{pushing2}, we have $Y_{\mu ,\alpha }=t_{\delta _1}^{\varepsilon _1}t_{\delta _2}^{\varepsilon _2}$ for some $\varepsilon _1$ and $\varepsilon _2\in \{ -1,1\}$, and by above arguments, $t_{c_1}$, $t_{c_2}\in G$ have lifts $h_1$ and $h_2\in \mathcal{M}(N_{g,n})$ with respect to $\psi $, respectively. Thus we have 
\begin{eqnarray*}
\psi (h_1^{\varepsilon _1}h_2^{\varepsilon _2})&=& t_{c_1}^{\varepsilon _1}t_{c_2}^{\varepsilon _2}\\
&\stackrel{(\text{V})}{=}& Y_{\mu ,\alpha }.
\end{eqnarray*}
Thus $h_1^{\varepsilon _1}h_2^{\varepsilon _2}\in \mathcal{M}(N_{g,n})$ is a lift of $Y_{\mu ,\alpha }\in G$ with respect to $\psi $ and $\psi :\mathcal{M}(N_{g,n})\rightarrow G$ is surjective. We have completed the proof of Theorem~\ref{main-thm}.

\end{proof}

\par
{\bf Acknowledgement: } The author would like to express his gratitude to Hisaaki Endo, for his encouragement and helpful advices. The author also wish to thank Susumu Hirose for his comments and helpful advices. The author was supported by JSPS KAKENHI Grant number 15J10066.


\begin{thebibliography}{99}
\bibitem{Birman-Chillingworth} 
J. S. Birman, D. R. J. Chillingworth, \emph{On the homeotopy group of a non-orientable surface}, Proc. Camb. Philos. Soc. \textbf{71} (1972), 437--448.

\bibitem{Epstein} 
D. B. A. Epstein, \emph{Curves on 2-manifolds and isotopies}, Acta Math. \textbf{115} (1966), 83-107.


\bibitem{Gervais}
S. Gervais, \emph{Presentation and central extensions of mapping class groups}, Trans. Amer. Math. Soc. \textbf{348} (1996), 3097--3132.

\bibitem{Gervais2}
S. Gervais, \emph{A finite presentation of the mapping class group of a punctured surface}, Topology \textbf{40} (2001), no. 4, 703--725.
 
\bibitem{Harer} 
L. Harer, \emph{The second homology group of the mapping class groups of orientable surfaces}, Invent. Math. 72, 221 239 (1983) 
 
\bibitem{Hatcher-Thurston}
A. Hatcher, W. Thurston, \emph{A presentation for the mapping class group of a closed orientable surface}, Top. \textbf{19} (1980), 221--237. 

\bibitem{Johnson}
D. L. Johnson, \emph{Presentations of Groups}, London Math. Soc. Stud. Texts \textbf{15} (1990). 
 

\bibitem{Korkmaz2}
M. Korkmaz, \emph{Mapping class groups of nonorientable surfaces}, Geom. Dedicata. \textbf{89} (2002), 109--133.

 
\bibitem{Labruere-Paris} 
C. Labru\`ere and L. Paris, \emph{Presentations for the punctured mapping class groups in terms of Artin groups}, Algebr. Geom. Topol. \textbf{1} (2001), 73--114. 
 
\bibitem{Lickorish1} 
W. B. R. Lickorish, \emph{Homeomorphisms of non-orientable two-manifolds}, Proc. Camb. Philos. Soc. \textbf{59} (1963), 307--317.
 
\bibitem{Lickorish2} 
W. B. R. Lickorish, \emph{On the homeomorphisms of a non-orientable surface}, Proc. Camb. Philos. Soc. \textbf{61} (1965), 61--64.

\bibitem{Luo} 
F. Luo, \emph{A presentation of the mapping class groups}, Math. Res. Lett. \textbf{4} (1997), 735--739. 

\bibitem{Paris-Szepietowski}
L. Paris and B. Szepietowski, \emph{A presentation for the mapping class group of a nonorientable surface}, Bull. Soc. Math. France \textbf{143} (2015), no. 3, 503--566.

\bibitem{Stukow1} 
M. Stukow, \emph{Dehn twists on nonorientable surfaces}, Fund. Math. \textbf{189} (2006), 117--147.
 
\bibitem{Stukow2}
M. Stukow, \emph{A finite presentation for the mapping class group of a nonorientable surface with Dehn twists and one crosscap slide as generators}, J. Pure Appl. Algebra \textbf{218} (2014), no. 12, 2226--2239.

\bibitem{Szepietowski1} 
B. Szepietowski. \emph{Crosscap slides and the level 2 mapping class group of a nonorientable surface}, Geom. Dedicata \textbf{160} (2012), 169--183.


 

\bibitem{Wajnryb}
B. Wajnryb, \emph{A simple presentation for the mapping class group of an orientable surface}, Israel J. Math. \textbf{45} (1989), 157--174.

\end{thebibliography}
\end{document}